\documentclass[10pt,leqno]{amsart}
\usepackage{graphicx}
\usepackage{setspace}
\usepackage{tikz}
\usepackage{doi}
\usepackage{indentfirst,csquotes}

\usepackage{amssymb,amsthm,amsmath}
\usepackage{xcolor,paralist,hyperref,titlesec,fancyhdr,etoolbox}
\newtheorem{theorem}{Theorem}[]
\newtheorem{definition}[theorem]{Definition}

\newtheorem{lemma}[theorem]{Lemma}

\newtheorem{corollary}[theorem]{Corollary}

\titleformat{\section}[display]{\normalfont\huge\bfseries\centering}{\centering\chaptertitlename\thechapter}{10pt}{\Large}
\titlespacing*{\section}{0pt}{0ex}{0ex}

\titleformat{\section}
  {\normalfont\Small\centering} % Font, size, and centering
  {\thesection.}{1em}{\MakeUppercase} % Automatic uppercase

\titlespacing*{\section}{0pt}{1.5ex plus 1ex minus .2ex}{1.3ex plus .2ex}

\hypersetup{ colorlinks=true, linkcolor=black, filecolor=black, urlcolor=black }

\usepackage{lipsum}
\begin{document}
\title{Finding Squares in a Product involving Squares} 
\author{Thang Pang Ern}
\address{Department of Mathematics, National University of Singapore, 10 Lower Kent Ridge Road, Singapore 119076}
\email{thangpangern@u.nus.edu}
\maketitle

\begin{abstract}
We wish to discuss positive integer solutions to the Diophantine equation \begin{align*}
    \prod_{k=1}^n(k^2+1)=b^2.
\end{align*}
Some methods in analytic number theory will be used to tackle this problem.
\end{abstract} %%%%%%%%%

\bigskip

\section{Introduction}
We motivate our discussion with a simple geometry problem mentioned on Numberphile in 2014. Arrange three congruent squares $ABFE$, $BCGF$, and $CDHG$ in the following manner (Figure \ref{fig:The three square geometry problem}). From vertex $D$, construct a line segment from it to the bottom-left vertex of each square — the vertices are namely $E$, $F$, and $G$. Let $\angle DEF=\alpha$, $\angle DFG=\beta$, and $\angle DGH=\gamma$. It can be proven using elementary geometry, that $\alpha+\beta+\gamma=\pi/2$.
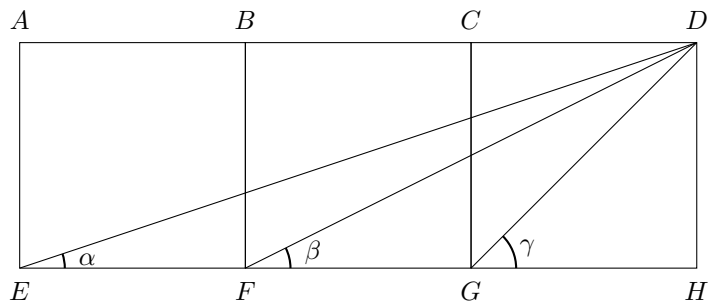
\begin{figure}[h!]
    \centering
    \begin{tikzpicture}[scale=1.5]

% Define points for the squares
\coordinate (A) at (0, 2);
\coordinate (B) at (2, 2);
\coordinate (C) at (4, 2);
\coordinate (D) at (6, 2);
\coordinate (E) at (0, 0);
\coordinate (F) at (2, 0);
\coordinate (G) at (4, 0);
\coordinate (H) at (6, 0);

% Draw squares
\draw (A) -- (B) -- (F) -- (E) -- cycle;
\draw (B) -- (C) -- (G) -- (F) -- cycle;
\draw (C) -- (D) -- (H) -- (G) -- cycle;

% Draw diagonal lines
\draw (E) -- (D);
\draw (F) -- (D);
\draw (G) -- (D);

% Label angles
\node at (0.6, 0.08) {$\alpha$};
\node at (2.6, 0.14) {$\beta$};
\node at (4.5, 0.18) {$\gamma$};

\draw[thick] (0.4,0) arc[start angle=0, end angle=18.4, radius=0.4];
\draw[thick] (2.4,0) arc[start angle=0, end angle=26.56, radius=0.4];
\draw[thick] (4.4,0) arc[start angle=0, end angle=45, radius=0.4];
% Label points
\node at (0, 2.2) {$A$};
\node at (2, 2.2) {$B$};
\node at (4, 2.2) {$C$};
\node at (6, 2.2) {$D$};
\node at (0, -0.2) {$E$};
\node at (2, -0.2) {$F$};
\node at (4, -0.2) {$G$};
\node at (6, -0.2) {$H$};

\end{tikzpicture}
    \caption{The three square geometry problem}
    \label{fig:The three square geometry problem}
\end{figure}
\newblock
\newline As $\triangle DGH$ is isosceles, it follows that $\gamma=\pi/4$. Thus, this identity is reduced to $\alpha+\beta=\pi/4$.
\newline
\newline In contrast, using complex numbers makes the problem far simpler. In particular, using arguments, it is easy to show that \begin{align}
    \arg \left(( 1+i)(2+i)(3+i) \right)=\arg(1+i)+\arg(2+i)+\arg(3+i)=\frac{\pi}{2}.
\end{align}
Naturally, we can extend the problem to $n$ squares. As such, which values of $n$, where $n$ is a positive integer, yield the sum of angles being some rational multiple of $\pi$? That is to say, does there exist $p,q\in\mathbb{Z}$, where $q\ne 0$, such that \begin{align}
    \sum_{k=1}^{n}\theta_k=\sum_{k=1}^{n}\operatorname{arctan}\left(1/k\right)=\frac{p}{q}\cdot \pi?
\end{align}
If the sum of angles $\theta_1+\ldots+\theta_n$ converges, then we have many cases to consider (other than the case when $n=3$). However, it turns out that the sum of angles diverges \cite{arctangent}. That is, \begin{align}
\lim_{n\rightarrow\infty}\sum_{k=1}^{n}{{{\arctan}\left( \frac{1}{k} \right)}}\quad\text{is a divergent series}. 
\end{align}
One can use the integral test to verify this. Moreover, attempts to tackle the problem by constructing a recurrence relation involving the addition formula of inverse tangent are futile as there are no \textit{nice} properties of the following formula: \begin{align}
    \arctan u+\arctan v=\arctan \left( \frac{u+v}{1-uv} \right)
\end{align}
\section{The Number-Theoretic Connection}
This question can be regarded as a number theoretic one. Note that for $1\le k\le n$, the modulus of the complex number $k+i$ is $\sqrt{k^2+1}$, and every $z\in\mathbb{C}$ can be expressed as $z=a+bi$, where $a,b\in\mathbb{R}$. If \begin{align}
    \prod_{k=1}^{n}(k+i)
\end{align} is purely imaginary, then $z=bi$ (i.e. $a=0$), and so $|z|=b$. Consider the product \begin{align}
    \prod_{k=1}^n{\sqrt{k^2+1}}
\end{align} which is the product of the moduli of the complex numbers $k+i$ for all $1 \le k \le n$. Hence, the problem is equivalent to finding all pairs of natural numbers $\left(b,n\right)$ such that \begin{align}
    \prod_{k=1}^n(k^2+1)=b^2.
\end{align}
We wish to find all solutions to the Diophantine equation in (7), and we assert that $(b,n)=(10,3)$ is the only solution. This would conclude that the product in (5) is purely imaginary if and only if $n=3$. 
\newline
\newline In this paper, we will discuss and expand on a proof given by Cilleruelo. It aims to make his proof more accessible by explaining the parts that were deemed trivial. Moreover, as shown in (79), we will obtain a better bound for our problem, reducing the time taken to seek all positive integer solutions to the Diophantine equation.
\newline
\newline We first discuss the following problem from the 17th China Western Mathematical Invitational Competition in 2017. It is from Problem 1 of Day 1 of the competition \cite{cwmo}. For any prime $p$, this result would help us find a bound for $p$ in terms of $n$. 
\begin{theorem}
    Let $p$ be a prime and $n$ be a positive integer such that $p^2$ divides \begin{align}
    \prod_{k=1}^n(k^2+1).
\end{align} Then, $p<2n$.
\end{theorem} \begin{proof}
Define $P_n$ to be the above product. Since $p^2\mid P_n$, then $p\mid P_n$. One case to consider is $p^2\mid (k^2+1)$ for some $1 \le k \le n$. Then, $p^2<n^2+1$ and hence, $p<\sqrt{n^2+1}<2n$. Suppose otherwise, then there exists $1 \le a<b\le n$ such that $p\mid (a^2+1)$ and $p\mid (b^2+1)$, and so, taking the difference, $p\mid (b+a)(b-a)$. As $a+b<2n$ and $b-a>0$, $b^2-a^2<2n$. Thus, $p$ is a divisor of some number less than $2n$ and the result follows. 
\end{proof}

\begin{theorem}\label{pn non square}
For $n>3$, \begin{align}
    P_n=\prod_{k=1}^n(k^2+1)\quad\text{is non-square}.
\end{align}
\end{theorem}
Suppose on the contrary that for $n>3$, $P_n$ is square. Then, we can write it as a product of some primes $p$, where $p<2n$, in the following form: \begin{align}
    P_n=\prod_{p<2n}p^{\alpha_p}, \text{ where }\alpha_p \in \mathbb{N}
\end{align} Note that \begin{align}
    P_n=\prod_{k=1}^n(k^2+1)>\prod_{k=1}^n{k^2}=(n!)^2.
\end{align} Consider writing $n!$ as \begin{align}
    n!=\prod_{p\le n}p^{\beta_p}
\end{align} since $n!$ is a product of all the natural numbers from 1 to itself inclusive, and there are some numbers which are composite and hence, can be written as the product of primes. This expression is apt when we compare it with the expression for $P_n$ in terms of $\alpha_p$ claimed in (10).
\newline
\newline We have the following as we make use of the property that the logarithm of a product is the sum of that logarithm:
	\begin{align}
   {{\left( \prod\limits_{p\le n}{{{p}^{{{\beta }_{p}}}}} \right)}^{2}}&<\prod\limits_{p<2n}{{{p}^{{{\alpha }_{p}}}}} \\ 
  2\log \left( \prod\limits_{p\le n}{{{p}^{{{\beta }_{p}}}}} \right) &<\log \left( \prod\limits_{p<2n}{{{p}^{{{\alpha }_{p}}}}} \right) \\ 
  \sum\limits_{p\le n}{{{\beta }_{p}}\log p} &<\frac{1}{2} \sum\limits_{p<2n}{{{\alpha }_{p}}\log p}  
\end{align}
\begin{corollary}
For $n \ge 1$, $\alpha_2=\lceil{n/2}\rceil$.
\end{corollary}
\begin{proof}
Since $k^2 \equiv 0\text{ or }1 \text{ }(\operatorname{mod}4)$, then $k^2+1\equiv 1 \text{ or }2 \text{ }(\operatorname{mod}4)$. The initial result can be easily derived by considering the cases where $k$ is odd or even. Note that $\alpha_2$ denotes the number of times $P_n$ can be divided by 2.
\newline
\newline If $k$ is odd, then $k^2+1$ is even, which is why for all odd $k$, $P_k$ contributes an additional multiplicity of 2 (other than odd multiples) as compared to $P_{k-1}$. Hence, for odd $k$, consider the pair $\left(P_k,P_{k+1}\right)$, where both their values of $\alpha_2$ are the same. As there are $n/2$ such pairs, then $\alpha_2=\lceil{n/2}\rceil$.
\end{proof}
Before we state and prove an important lemma (Lemma \ref{important lemma}), we define the Legendre symbol and state a major theorem known as the law of quadratic reciprocity.
\begin{definition}[Legendre symbol]
Let $p$ be an odd prime. An integer $a$ is a quadratic residue modulo $p$ if it is congruent to a perfect square modulo $p$ and is a quadratic non-residue modulo $p$ otherwise. The Legendre symbol is a function of $a$ and $p$ defined as \[\left( \frac{a}{p} \right)=\left\{ \begin{array}{*{35}{l}}
   1 & \text{if }a\text{ is a quadratic residue modulo }p\text{ and }a\not\equiv 0\text{ }\left( \operatorname{mod}p \right),  \\
   -1 & \text{if }a\text{ is a non-quadratic residue modulo }p,  \\
   0 & \text{if }a\equiv 0\text{ }\left( \operatorname{mod}p \right).  \\
\end{array} \right.\]
\end{definition}
\begin{theorem}[law of quadratic reciprocity]\label{law of quadratic reciprocity} For distinct odd primes $p$ and $q$, \begin{align}
    \left({\frac {p}{q}}\right)\left({\frac {q}{p}}\right)=(-1)^{{\frac {p-1}{2}}\cdot {\frac {q-1}{2}}}.
\end{align}
\end{theorem}
\begin{lemma}\label{important lemma}
Let $p$ be an odd prime. If $p\text{ }|\text{ }(k^2+1)$ for $1 \le k \le n$, then $p\equiv 1 \text{ }(\operatorname{mod}4)$.
\end{lemma}
\begin{proof}
As $k^2\equiv -1\text{ }(\operatorname{mod}p)$, $-1$ is a quadratic residue modulo $p$. Here, we set $a=-1$ and so \[\left(\frac{-1}{p}\right)=1.\] By the law of quadratic reciprocity (Theorem \ref{law of quadratic reciprocity}), $p \equiv 1 \text{ }(\operatorname{mod} 4)$.
\end{proof}
\begin{lemma}
For any prime $p$, $x^2\equiv 0\text{ }(\operatorname{mod}p)$ has the unique solution $x\equiv 0 \text{ }(\operatorname{mod}p)$.
\end{lemma}
\begin{proof}
The only zero divisor in the ring $\mathbb{Z}/p\mathbb{Z}$ is 0. As such, if a product is 0, one of the factors must be 0 as well, and the result follows.
\end{proof}
\begin{theorem}\label{important theorem pj}
For any prime $p$, each interval of length $p^j$ contains two solutions to the congruence $x^2\equiv 1 \text{ }\left(\operatorname{mod}p^j\right)$ \cite{virginia}.
\end{theorem}
\begin{proof}
We proceed by induction. First, we solve the base case $x^2\equiv 1\text{ }(\operatorname{mod}p)$. Then, our induction hypothesis would be for $j>1$, given a solution to $x^2\equiv 1\text{ }\left(\operatorname{mod}p^j\right)$, we find a solution to $x^2\equiv 1\text{ }\left(\operatorname{mod}p^{j+1}\right)$. The existence of the two and only two solutions, say $x_1$ and $-x_1$, is a consequence of the law of quadratic reciprocity. For this to occur, we must have \begin{align}
    \left(\frac{1}{p}\right)=1.
\end{align}
The base case is an immediate consequence of the law of quadratic reciprocity too.
\newline
\newline Assume that there exists a solution $x_0$ such that $x_0^2\equiv 1 \text{ }\left(\operatorname{mod}p^j\right)$. We wish to find a lift of $x_0 \text{ }\left(\operatorname{mod}p^j\right)$ to $x_1\text{ }\left(\operatorname{mod}p^{j+1}\right)$ that satisfies $x_1^2\equiv 1\text{ }\left(\operatorname{mod}p^{j+1}\right)$. To ensure that $x_1\text{ }\left(\operatorname{mod}p^{j+1}\right)$ is indeed a lift of $x_0 \text{ }\left(\operatorname{mod}p^j\right)$, consider \begin{align}
    x_1=x_0+p^jy_0.
\end{align} Squaring both sides, \begin{align}
 x_1^2&=x_0^2+2p^jx_0y_0+p^{j+1}y_0   \\
 &\equiv x_0^2+2p^jx_0y_0 \text{ }\left(\operatorname{mod}p^{j+1}\right)\\
 x_0^2+2p^jx_0y_0 &\equiv 1 \text{ }\left(\operatorname{mod}p^{j+1}\right)
\end{align}
Since $x_0^2-1$ is a multiple of $p^j$, setting $x_0^2-1=\lambda p^j$, (21) implies $\lambda p^j+2p^jx_0y_0$ is a multiple of $p^{j+1}$. Since $\operatorname{gcd}(2x_0,p)=1$, then $\operatorname{gcd}(2p^jx_0,p^{j+1})=p^j$ and hence, $\lambda+2x_0y_0$ is a multiple of $p$. As such, there exists a unique solution to \begin{align}
    2x_0y_0\equiv \frac{1-x_0^2}{p^j}\text{ }(\operatorname{mod}p)
\end{align} and so \begin{align}
    y_0\equiv\frac{1-x_0^2}{2p^jx_0}\text{ }(\operatorname{mod}p).
\end{align}
The existence of the two solutions follows by the law of quadratic reciprocity.
\end{proof} 
\begin{lemma}\label{alpha beta first bounds lemma}
If $p\equiv 1\text{ }(\operatorname{mod}4)$, then \begin{align}
    \frac{1}{2}{{\alpha }_{p}}-{{\beta }_{p}}\le \sum\limits_{j\le \log n/\log p}{\left( \left\lceil \frac{n}{{{p}^{j}}} \right\rceil -\left\lfloor \frac{n}{{{p}^{j}}} \right\rfloor  \right)}+\sum\limits_{\log n/\log p<j\le \log \left( {{n}^{2}}+1 \right)/p}{\left\lceil \frac{n}{{{p}^{j}}} \right\rceil }.
\end{align}
\end{lemma}
\begin{proof}
Let $S_1$ and $S_2$ denote the following sets: \begin{align}
    S_1=\left\{1\le k\le n,\text{ }p\text{ prime}: p^j\mid (k^2+1)\right\} \text{ and }S_2=\left\{1\le k\le n,\text{ }p\text{ prime}:p^j\mid k\right\}
\end{align}
Then, using Theorem \ref{important theorem pj}, we can find an upper bound for $\alpha_p$ and an equation for $\beta_p$. That is, \begin{align}
    \alpha_p=\sum_{j\le \log(n^2+1)/\log p}|S_1| \le \sum_{j\le \log(n^2+1)/\log p} 2\lceil{n/p^j}\rceil
\end{align} and \begin{align}
    \beta_p=\sum_{j\le \log n/\log p}|S_2| = \sum_{j\le \log n/\log p}\left\lfloor{\frac{n}{p^j}}\right\rfloor.
\end{align}
Hence, \begin{align}
   \frac{1}{2}{{\alpha }_{p}}-{{\beta }_{p}}&\le \sum\limits_{j\le \log ({{n}^{2}}+1)/\log p}{\left\lceil {\frac{n}{p^j}}  \right\rceil }-\sum\limits_{j\le \log n/\log p}{\left\lfloor {\frac{n}{p^j}} \right\rfloor } \\ 
 & =\sum\limits_{j\le \log n/\log p}{\left\lceil {\frac{n}{p^j}} \right\rceil }+\sum\limits_{\log n/\log p\le j\le \log ({{n}^{2}}+1)/\log p}{\left\lceil {\frac{n}{p^j}}  \right\rceil }-\sum\limits_{j\le \log n/\log p}{\left\lfloor {\frac{n}{p^j}} \right\rfloor } \\ 
 & =\sum\limits_{j\le \log n/\log p}{\left( \left\lceil {\frac{n}{p^j}}  \right\rceil -\left\lfloor {\frac{n}{p^j}}  \right\rfloor  \right)}+\sum\limits_{\log n/\log p\le j\le \log ({{n}^{2}}+1)/\log p}{\left\lceil {\frac{n}{p^j}}  \right\rceil }  
\end{align} and the result follows.
\end{proof}
We observe a nice property related to the floor and ceiling function (Lemma \ref{floor and ceiling function lemma}) before stating an important theorem (Theorem \ref{second important theorem}).
\begin{lemma}\label{floor and ceiling function lemma}
Let $C=\lceil{x}\rceil-\lfloor{x}\rfloor$. If $x\in\mathbb{Z}$, $C=0$, otherwise, $C=1$.
\end{lemma}
\begin{proof}
If $x\in\mathbb{Z}$, $\lceil{x}\rceil=\lfloor{x}\rfloor=x$, and so $C=0$. Suppose $x\in\mathbb{R}\backslash\mathbb{Z}$. By definition, $m < \lceil{x}\rceil \le m+1$ and $m \le \lfloor{x}\rceil < m+1$, where $m\in\mathbb{Z}$. Hence, $-m-1<-\lceil{x}\rceil \le -m$, and so $-1< C \le 1$. Since $C$ is an integer, $0 \le C \le 1$. However, $\lfloor{x}\rfloor\ne \lceil{x}\rceil$, which asserts that $C=1$. 
\end{proof}
\begin{theorem}\label{second important theorem}
If $p \equiv 1 \text{ }(\operatorname{mod}4)$, then \begin{align}
    \frac{1}{2}{{\alpha }_{p}}-{{\beta }_{p}}\le \frac{\log \left( {{n}^{2}}+1 \right)}{\log p}.
\end{align}
\end{theorem}
\begin{proof}
The proof hinges on Lemma \ref{alpha beta first bounds lemma} and Lemma \ref{floor and ceiling function lemma}. Note that $\lceil{n/p^j}\rceil-\lfloor{n/p^j}\rfloor\le 1\text{ }$ for all $j \le \log n/\log p$. Hence, \begin{align}
    \sum\limits_{j\le \log n/\log p}{\left( \left\lceil {\frac{n}{p^j}}  \right\rceil -\left\lfloor {\frac{n}{p^j}}  \right\rfloor  \right)} \le \sum\limits_{j\le \log n/\log p}.
\end{align} It is also clear that for all $\log n/\log p \le j \le \log(n^2+1)/\log p$, $\lceil{n/p^j}\rceil \le 1$. As such, \begin{align}
    \sum\limits_{\log n/\log p\le j\le \log ({{n}^{2}}+1)/\log p}{\left\lceil \frac{n}{p^j} \right\rceil } \le \sum\limits_{\log n/\log p\le j\le \log ({{n}^{2}}+1)/\log p}.
\end{align} Combining both inequalities, we have \begin{align}
\frac{1}{2}{{\alpha }_{p}}-{{\beta }_{p}}&\le \sum\limits_{j\le \log n/\log p}+\sum\limits_{\log n/\log p\le j\le \log ({{n}^{2}}+1)/\log p} \\&= \sum_{1 \le j \le \log(n^2+1)/\log p} \\&\le \frac{\log \left( {{n}^{2}}+1 \right)}{\log p}    
\end{align} and we are done.
\end{proof}
Returning to the inequality established in (15), if we can establish that $P_n$ is never a perfect square for $n$ greater than or equal to some number, say $\lambda$, then we are done. Consequently, we just need to verify the initial claim for $1\le n \le \lambda-1$.
\section{Some Results in Analytic Number Theory}
We shall make use of two well-known ideas, which are namely the prime-counting function and Bertrand's postulate (Theorem \ref{bertrand's postulate}).
\begin{definition}[prime counting function]
Let $\pi(x)$ denote the number of primes $p$ less than or equal to $x$.
\end{definition}
\begin{theorem}[prime number theorem]
A well-known asymptotic relation for $\pi(n)$ is as follows: for large $n$, \begin{align}
    \pi(n) \sim \frac{n}{\log n}
\end{align}
\end{theorem}
\begin{theorem}[Bertrand's postulate]\label{bertrand's postulate}
For every $n>1$, there exists at least one prime $p$ such that $n<p<2n$.
\end{theorem}
We make use of the prime-counting function and Theorem \ref{second important theorem} to find an upper bound for the sum of $\beta_p\log p$ in terms of $\alpha_p\log p$, where the sum is taken over appropriate values of $p$. We will state this in a while.
\begin{definition}[modified prime counting function]
Let $\pi(n,a,b)$ denote the number of primes less than or equal to $n$ which satisfies $n \equiv a \text{ }(\operatorname{mod}b)$.
\end{definition}
\begin{theorem}\label{theorem will not prove}
\begin{align}
    \sum\limits_{\begin{smallmatrix} 
 p\le n \\ 
 p\not{\equiv }1\text{ }(\operatorname{mod}4) 
\end{smallmatrix}}{{{\beta }_{p}}}\log p\le \frac{1}{2}\left\lceil \frac{n}{2} \right\rceil \log 2+\log ({{n}^{2}}+1)\pi (n,1,4)+\frac{1}{2}\sum\limits_{n<p<2n}{{{\alpha }_{p}}\log p}
\end{align}
\end{theorem}
We will not prove Theorem \ref{theorem will not prove} in detail. The definition of $\pi(n,a,b)$ is crucial here since if $p \equiv 1\text{ }(\operatorname{mod}4)$, then \begin{align}
    \operatorname{log}p\sum_{\begin{smallmatrix}p\le n \\ p\equiv 1\text{ }(\operatorname{mod}4)
    \end{smallmatrix}}\left(\frac{1}{2}\alpha_p-\beta_p\right)\le \operatorname{log}\left(n^2+1\right)\pi(n,1,4).
\end{align} Also, note that $\left\lfloor n/2\right\rfloor\operatorname{log}2/2$ is just $\alpha_2\operatorname{log}2/2$, and the open interval $(n,2n)$ can be expressed as $(1,2n) \backslash (1,n]$.
\begin{lemma}\label{estimate 1 crucial}
If $p>n$, then $\alpha_p\le 2$.
\end{lemma}
\begin{proof}
    We established the inequality \begin{align}
        \alpha_p \le \sum_{j\le \operatorname{log}(n^2+1)/\operatorname{log}p} 2\left\lceil{\frac{n}{p^j} }\right\rceil
    \end{align} in (26). If $p>n$, then $n/p<1$. A key observation is that \begin{align}
        \operatorname{log}(n^2+1) < 2\operatorname{log}p,
    \end{align}
    which is a consequence of $n^2+1 < p^2$ (the inequality sign does not change because the logarithmic function is strictly increasing). The solution to this inequality yields \begin{align}
        p^2-n^2 &>1\\
        (p+n)(p-n)&> 1
    \end{align}
    and this is justified because $p+n$ is positive, and $p-n$ is also positive (because $p>n$). As such, for $j\le \operatorname{log}(n^2+1)/\operatorname{log}p$, it would imply that $j<2$, so there is only one value of $j$ to consider in the sum for the case where $p>n$, and that is $j=1$. We briefly discussed this earlier as $n/p<1$, implying that $\left\lceil n/p\right\rceil \le 1$.
\end{proof}
\begin{lemma}\label{estimate 2 crucial}
If $p\le n$, then \begin{align}
    \beta_p \ge\frac{n-p}{p-1}-\frac{\operatorname{log}n}{\operatorname{log}p}\ge \frac{n-1}{p-1}-\frac{\operatorname{log}(n^2+1)}{\operatorname{log}p}.
\end{align}
\end{lemma}
\begin{proof}
    Recall that (27) yields the equality \begin{align}
        \beta_p =\sum_{j\le \operatorname{log}n/\operatorname{log}p}\left\lfloor n/p^j\right\rfloor.
    \end{align}
    It is worth noting that \begin{align}
        \frac{n-p}{p-1}&=\frac{n}{p-1}-\frac{p}{p-1}.
    \end{align}
Hence,    
\begin{align}
   {{\beta }_{p}}&=\sum\limits_{j\le \log n/\log p}{\left\lfloor n/{{p}^{j}} \right\rfloor } \\ 
 & \ge \sum\limits_{j\le \log n/\log p}{\frac{n}{{{p}^{j}}}}-\sum\limits_{j\le \log n/\log p}{{}} \\ 
 & =n\sum\limits_{j\le \log n/\log p}{\frac{1}{{{p}^{j}}}}-\frac{\log n}{\log p} \\ 
 & =n\sum\limits_{j=1}^{\left\lfloor \log n/\log p \right\rfloor }{{{p}^{-j}}}-\frac{\log n}{\log p} \\ 
 & =\frac{n}{p-1}\left( 1-\frac{1}{{{p}^{\left\lfloor \log n/\log p \right\rfloor }}} \right)-\frac{\log n}{\log p} \\ 
 & =\frac{n}{p-1}-\frac{n}{{{p}^{\left\lfloor \log n/\log p \right\rfloor }}\left( p-1 \right)}-\frac{\log n}{\log p}  
\end{align}
What is left to show is \begin{align}
    \frac{n}{{{p}^{\left\lfloor \log n/\log p \right\rfloor }}\left( p-1 \right)}\le \frac{p}{p-1}.
\end{align}
In other words, \begin{align}
    {{p}^{\left\lfloor \log n/\log p \right\rfloor +1}}\ge n,
\end{align}
which is clear because \begin{align}
    \left\lfloor \frac{\operatorname{log}n}{\operatorname{log}p}\right\rfloor+1 \ge \frac{\operatorname{log}n}{\operatorname{log}p}
\end{align}
and $p^{\operatorname{log}n/\operatorname{log}p}=n$. With these, the lower bound \begin{align}
    \beta_p \ge\frac{n-p}{p-1}-\frac{\operatorname{log}n}{\operatorname{log}p}
\end{align} 
is obtained.
\newline
\newline Now, we show that \begin{align}
    \frac{n-p}{p-1}-\frac{\operatorname{log}n}{\operatorname{log}p}\ge \frac{n-1}{p-1}-\frac{\operatorname{log}(n^2+1)}{\operatorname{log}p}.
\end{align}
Working backwards, \begin{align}
   \frac{\log ({{n}^{2}}+1)}{\log p} &\ge 1+\frac{\log n}{\log p} \\ 
 \log ({{n}^{2}}+1) &\ge \log pn \\ 
 {{n}^{2}}-pn+1 &\ge 1 \\ 
 n\left( n-p \right) &\ge 0  
\end{align}
For the quadratic inequality $n(n-p)\ge 0$ in (61), since $n$ is always positive, then $n-p\ge 0$, or equivalently, $n\ge p$, which is the condition stated in the lemma.
\end{proof}
\begin{lemma}
\begin{align}
    \left( n-1 \right)\sum\limits_{\begin{smallmatrix} 
 p\le n \\ 
 p\not{\equiv }1\text{ }\left( \operatorname{mod} 4 \right) 
\end{smallmatrix}}{\frac{\log p}{p-1}}< \left( n+1 \right)\frac{\log 2}{4}+\log \left( {{n}^{2}}+1 \right)\pi \left( n \right)+\sum\limits_{n<p<2n}{\log p}
\end{align}
\end{lemma}
\begin{proof}
    We apply the estimates in Lemmas \ref{estimate 1 crucial} and \ref{estimate 2 crucial} to Theorem \ref{theorem will not prove}, so 
    \begin{align}
        \sum\limits_{\begin{smallmatrix}
   p\le n  \\
   p\not{\equiv }1\text{ }(\operatorname{mod} 4)  \\
\end{smallmatrix}}{\left( \frac{n-1}{p-1}-\frac{\log ({{n}^{2}}+1)}{\log p} \right)\log p}&<\frac{1}{2}\left\lceil \frac{n}{2} \right\rceil \log 2+\log ({{n}^{2}}+1)\pi (n,1,4)+\sum\limits_{n<p<2n}{\log p}.
    \end{align}
As such, a strict upper bound for \begin{align}
    \left( n-1 \right)\sum_{\begin{smallmatrix}
       p\le n\\ p\not\equiv 1\text{ }(\operatorname{mod} 4)
   \end{smallmatrix}}\frac{\log p}{p-1}-{\sum_{\begin{smallmatrix}
       p\le n\\ p\not\equiv 1\text{ }(\operatorname{mod} 4)
   \end{smallmatrix}}}\log ({{n}^{2}}+1)
\end{align} is \begin{align}
    \frac{1}{2}\left\lceil \frac{n}{2} \right\rceil \log 2+\log ({{n}^{2}}+1)\pi (n,1,4)+\sum\limits_{n<p<2n}{\log p}
\end{align}
and (39) evaluates to 
\begin{align}
\left( n+1 \right)\frac{\log 2}{4}+\log ({{n}^{2}}+1)\pi (n)+\sum\limits_{n<p<2n}{\log p},
\end{align}
yielding the desired result.
\end{proof}
To wrap up the proof, we exploit Chebyshev's functions. The first and second Chebyshev Functions are denoted by $\vartheta(n)$ and $\psi(n)$ respectively.
\begin{definition}[von Mangoldt function]
    For every integer $n\ge 1$, define \[\Lambda \left( n \right)=\left\{ \begin{array}{*{35}{l}}
   \log p & \text{if }n={{p}^{m}}\text{ for some prime }p\text{ and some }m\ge 1,  \\
   0 & \text{otherwise}\text{.}  \\
\end{array} \right.\]
\end{definition}
\begin{definition}[first Chebyshev function]
For any real number $n\ge 1$, \begin{align}
        \vartheta(n)=\sum_{p\le n}\operatorname{log}p.
    \end{align}
\end{definition}
\begin{definition}[second Chebyshev function]
For any real number $n\ge 1$, \begin{align}
        \psi(x)=\sum_{n\le x}\Lambda(n)=\sum_{p^m\le x}\operatorname{log}p.
    \end{align}
\end{definition}
\begin{lemma}\label{theta bounded by psi lemma}
For all $n\ge 1$, $\vartheta(n) \le \psi(n)$ \cite{apostol}.
\end{lemma}
\begin{proof}
It is clear by the definitions of $\vartheta(n)$ and $\psi(n)$ that \begin{align}
    \psi (n)=\vartheta (n)+\vartheta \left( \sqrt{n} \right)+\vartheta \left( \sqrt[3]{n} \right)+\ldots 
\end{align}
so the result follows.
\end{proof}
\begin{corollary}\cite{chebyinequality}
    If $p$ is prime, then \begin{align}
        \vartheta(n)\le \pi(n)\operatorname{log}n.
    \end{align}
\end{corollary}
\begin{proof}
Using Lemma \ref{theta bounded by psi lemma}, as $\vartheta(n)\le \psi(n)$, by definition of $\psi(n)$, we have \begin{align}
    \psi(n)=\sum_{p\le n}\left\lfloor\frac{\operatorname{log}n}{\operatorname{log}p}\right\rfloor \cdot \operatorname{log}p\le \sum_{p\le n}\frac{\operatorname{log}n}{\operatorname{log}p}\cdot \operatorname{log}p=\sum_{p\le n}1 \cdot \operatorname{log}n=\pi(n)\operatorname{log}n
\end{align}
This shows that $\vartheta(n)$ and $\pi(n)$ are closely related.
\end{proof}
\begin{theorem}
    If $p$ is prime, then \begin{align}
        \sum_{n<p<2n}\operatorname{log}p\le \vartheta(2n-1).
    \end{align}
\end{theorem}
\begin{proof}
The sum of $\operatorname{log}p$ over the interval $n<p<2n$ is $\vartheta(2n-1)-\vartheta(n)$, and the result follows.
\end{proof}
\begin{theorem} \cite{erdos} For all $\varepsilon>0$, there exists $M>0$ such that for all $n>M$,
    \begin{align}
        \pi(n)<(1+\varepsilon)\frac{n}{\operatorname{log}n}.
    \end{align}
\end{theorem}
\begin{theorem}\label{upper bound for logpp-1 p equiv 3 mod 4}
An upper bound for $\operatorname{log}p/(p-1)$ is established \cite{cilleruelo}, and it is given by the inequality
\begin{align}
    \sum_{\begin{smallmatrix}p\le n \\ p\not\equiv 1\text{ }(\operatorname{mod}4)\end{smallmatrix}}\frac{\operatorname{log}p}{p-1}<4+\frac{\operatorname{log}2}{4}.
\end{align}
\end{theorem}
\begin{proof}
\begin{align}
   \sum_{\begin{smallmatrix}
   p\le n  \\
   p\not{\equiv }1\text{ }\left( \operatorname{mod} 4 \right)  \\
\end{smallmatrix}}\frac{\log p}{p-1}&<\frac{n+1}{n-1}\cdot \frac{\log 2}{4}+\frac{\log \left( {{n}^{2}}+1 \right)}{n-1}\pi \left( n \right)+\frac{1}{n-1}\sum\limits_{n<p<2n}{\log p} \\ 
 & <\frac{n+1}{n-1}\cdot \frac{\log 2}{4}+(1+\varepsilon )\cdot \frac{n}{n-1}\cdot \frac{\log \left( {{n}^{2}}+1 \right)}{\log n}+\frac{1}{n-1}\vartheta (2n-1) \\ 
 & <\frac{n+1}{n-1}\cdot \frac{\log 2}{4}+(1+\varepsilon )\cdot \frac{n}{n-1}\cdot \frac{\log \left( {{n}^{2}}+1 \right)}{\log n}+\frac{1}{n-1}(1+\varepsilon )\frac{2n-1}{\log \left( 2n-1 \right)}\cdot \log \left( 2n-1 \right) \\ 
 & =\frac{n+1}{n-1}\cdot \frac{\log 2}{4}+(1+\varepsilon )\cdot \frac{n}{n-1}\cdot \frac{\log \left( {{n}^{2}}+1 \right)}{\log n}+(1+\varepsilon )\cdot \frac{2n-1}{n-1}
\end{align}
For large $n$, (78) evaluates to (L'Hôpital's Rule is useful here) \begin{align}
    \frac{\log 2}{4}+4(1+\varepsilon ).
\end{align}
Taking $\varepsilon>0$ to be arbitrarily small, the result follows.
\end{proof}
In Cilleruelo's proof, he cited different bounds used by Hardy and Wright in `An Introduction to the Theory of Numbers'. We obtained a better bound as mentioned in (79) in our proof although it does not affect the computation. At this stage, note that $4+\operatorname{log}2/4\approx 4.1732$. For $n\ge 1831$, we obtain the inequality \begin{align}
    \sum_{\begin{smallmatrix}
        p\le n\\ p \not\equiv 1\text{ }(\operatorname{mod}4)
    \end{smallmatrix}}\frac{\operatorname{log}p}{p-1}>4.1732,
\end{align}
so we have verified Theorem \ref{pn non square} for $n\ge 1831$.
\newline
\newline It is worth appreciating the following fact about the function $\operatorname{log}p/(p-1)$. 
\begin{theorem}\label{big o theorem}
    We have \begin{align}
        \sum_{p\le n}\frac{\operatorname{log}p}{p-1}=\operatorname{log}n+\mathcal{O}(1),
    \end{align}
    where $\mathcal{O}$ denotes the Big O Notation.
\end{theorem}
\newpage
Theorem \ref{big o theorem} implies that the execution time of this computation is independent of the size of the input. We provide a proof of this theorem.
\newline
\begin{proof} We have
    \begin{align}
   \sum\limits_{p\le n}{\frac{\log p}{p-1}}-\frac{\log p}{p}&=\sum\limits_{p\le n}{\log p\left( \frac{1}{p\left( p-1 \right)} \right)} \ll \sum\limits_{p\le n}{\frac{\log n}{n\left( n-1 \right)}} \ll 1.  
\end{align}
Hence, \begin{align}
    \sum_{p\le n}\frac{\operatorname{log}p}{p-1}=\sum_{p\le n}\frac{\operatorname{log}p}{p}
\end{align}
and this is bounded above by $\operatorname{log}n+\mathcal{O}(1)$.
\end{proof}
\begin{corollary}\label{logpp-1 sum diverges}
    The sum \begin{align}
        \sum_{p\le n}\frac{\operatorname{log}p}{p-1}
    \end{align}
    diverges for $p\ge 2$.
\end{corollary}
Even though it seems that the growth rate of the sum of $\operatorname{log}p/(p-1)$ is quite slow (here, we did not impose the condition that $p \not\equiv 1\text{ }(\operatorname{mod}4)$ unlike Theorem \ref{upper bound for logpp-1 p equiv 3 mod 4}), it is interesting that this sum over all $p\ge 2$ diverges, which is analogous to the case of the harmonic series. Moreover, if one were to attempt to prove Corollary \ref{logpp-1 sum diverges} using the integral test, it would be futile. Back to $\operatorname{log}p/(p-1)$, if we take the sum of all such primes $p$ such that they are not congruent to 1 modulo 4, it actually converges but at a very slow rate. Recall that convergence was established in (79).
\begin{definition}[harmonic series]
    The harmonic series is defined to be the infinite series \begin{align}
        \sum_{n=1}^{\infty}\frac{1}{n}=\frac{1}{1}+\frac{1}{2}+\frac{1}{3}+\hdots
    \end{align}
\end{definition}
\begin{theorem}
    The harmonic series is divergent.
\end{theorem}
Now, it suffices to prove the remaining theorem.
\begin{theorem}
    $P_n$ is not a perfect square for $4\le n \le 1830$.
\end{theorem}
\begin{proof}
    We use Lemma \ref{important lemma} and deal with primes that are congruent to 1 modulo 4. Note that $17=4^2+1$ and $17\equiv 1\text{ }(\operatorname{mod}4)$. The next time 17 divides $k^2+1$ occurs when $k=17-4=13$. Thus, $P_n$ is not a perfect square for $4 \le n \le 12$. Next, $101=10^2+1$ and $101 \equiv 1\text{ }(\operatorname{mod}4)$. The next time 101 divides $k^2+1$ occurs when $k=101-10=91$, so $P_n$ is not a perfect square for $10 \le n \le 90$. Next, $1297=36^2+1$. The next time 1297 divides $k^2+1$ occurs when $k=1297-36=1261$. Hence, $P_n$ is not a square for $36\le n \le 1260$.
    \newline
    \newline What is left to show is that $P_n$ is not a perfect square for $1260 \le n \le 1830$. This can be easily shown.
\end{proof}
To conclude, Theorem \ref{pn non square} holds true and the only solution to the diophantine equation in (7) is $(b,n)=(10,3)$.

\end{document}